\documentclass[12pt]{amsart}
\usepackage{amsmath,amssymb}
\usepackage{amsfonts}
\usepackage{amsthm}
\usepackage{latexsym}
\usepackage{graphicx}

\def\p{\partial}

\def\R{\mathbb{R}}
\def\C{\mathbb{C}}

\def\vv<#1>{\langle#1\rangle}
\def\ol{\overline}

\def\s1{{\mathbb{S}^1}}

\def\XXint#1#2{\setbox0=\hbox{$#1{#2}{\int}$}{#2}\kern-.5\wd0 }

\def\XXint#1#2#3{{\setbox0=\hbox{$#1{#2#3}{\int}$}
     \vcenter{\hbox{$#2#3$}}\kern-.5\wd0}}

\def\vv<#1>{{\left\langle#1\right\rangle}}

\def\sph{\mathbb{S}}
\def\conf{{\rm Conf}}
\def\diag{{\rm diag}}
\newtheorem{thm}{Theorem}[section]

\newtheorem{lem}{Lemma}[section]

\theoremstyle{definition}

\theoremstyle{remark}

\newtheorem{rem}{Remark}[section]

\numberwithin{equation}{section}

\begin{document}
\title{Rigidity of isometric immersions into the light cone }
\author{Jian-Liang Liu$^1$}
\address{Department of Mathematics, Shantou University, Shantou, Guangdong, 515063, China}
\email{liujl@stu.edu.cn}
\thanks{$^1$Research partially supported by the China Postdoctoral Science Foundation 2016M602497 and NSFC 61601275. }
\author{Chengjie Yu$^2$}
\address{Department of Mathematics, Shantou University, Shantou, Guangdong, 515063, China}
\email{cjyu@stu.edu.cn}
\thanks{$^2$Research partially supported by NSFC 11571215 and the Yangfan project of the Guangdong Province.}
\renewcommand{\subjclassname}{%
  \textup{2010} Mathematics Subject Classification}
\date{}
\begin{abstract}
In this paper, we show the rigidity of isometric immersions for a Riemannian manifold of dimension $n-1$ into the light cone of $n+1$ dimensional  Minkowski, de Sitter  and anti-de Sitter spacetimes for $n\geq 3$.
\end{abstract}
\maketitle\markboth{Liu \& Yu}{Isometric immersions into the light cone}
\section{Introduction}
Isometric embedding into Euclidean spaces is a classical subject in differential geometry. For examples, the Weyl problem (see \cite{Ni,Po}), Nash's embedding theorem (see \cite{Na}) etc. The book \cite{HH} by Han and Hong makes an excellent exposition of the subject, especially for isometric embeddings of surfaces into $\R^3$. Because of applications in  General Relativity, isometric embeddings of Riemannian manifolds into Minkowski spacetime were also studied by experts. For example, the works \cite{BS,CY,GJ} considered isometric embedding of codimension 1 into Minkowski spacetime. In \cite{BLY,Epp,Lau,SC,WY1,WY2}, the authors defined quasi-local mass/energy by isometric embedding of a surface into $\R^{1,3}$. Previously, quasi-local mass was defined by isometric embedding into $\R^3$ in \cite{BY,Ki,LY} . Note that isometric embedding of surfaces into $\R^{1,3}$ is of codimension 2 and hence lack of rigidity. To handle this problem, in \cite{SC,WY2}, the authors consider critical isometric embeddings of certain energy functionals. There are other interesting proposals to handle this problem. In \cite{Epp}, Epp proposed to impose one more restriction $F=F_{\rm ref}$ on the isometric embedding to handle this problem, where $F$ and $F_{\rm ref}$ are essentially the normal curvatures of the surface in the physical spacetime and the reference spacetime respectively. On the other hand, in \cite{BLY,Lau}, the authors considered isometric embedding of surfaces into the light cone of $\R^{1,3}$, so that reduced the isometric embedding problem to the codimension 1 case. However, according to the authors's knowledge, rigidity of isometric embeddings in the two proposals has not been proved.

The existence of isometric embeddings into the light cone of $\R^{1,3}$ was first proved by Brinkmann \cite{Br} (see also \cite{AD,PR,LJ}). Brinkmann showed that an $(n-1)$-dimensional Riemannian manifold $(M,g)$ can be locally embedded into the light cone of $\R^{1,n}$ if and only if $(M,g)$ is locally conformally flat. In fact, Brinkmann wrote down the isometric embedding explicitly by using the conformal factor.  In this paper, we prove that the isometric embedding constructed by Brinkmann is essentially the unique isometric embedding into the light cone. Indeed, we also extend Brinkmann's result and rigidity of isometric embeddings to light cones of de Sitter and anti-de Sitter spacetimes.

For the Minkowski spacetime $\R^{1,n}$, let $o^{(0)}$ be the origin, then the future light cone at $o^{(0)}$ (the set forms by future directed null geodesics starting at $o^{(0)}$) is
\begin{equation}
\mathcal L_+^{(0)}=\{(t,x)\in \R^{1,n}\ |\ -t^2+\|x\|^2=0\ {\rm and}\ t>0\}.
\end{equation}
The $n+1$ dimensional de Sitter spacetime $dS_{n+1}$ can be defined as the hypersurface given by
\begin{equation}
-t^2+x_1^2+x_2^2+\cdots+x_{n+1}^2=1.
\end{equation}
 of $\R^{1,{n+1}}$ equipped with the induced metric. The Lorentz group $O(1,n+2)$ acts on $dS_{n+1}$ transitively. Let $o^{(1)}=(0,0,\cdots,1)$. Then, the isotropy group at $o^{(1)}$ of the action is $O(1,n)$ by identifying $A\in O(1,n)$ with
 \begin{equation}\label{eqn-A-1}
 A^{(1)}=\left(\begin{array}{ll}A&0\\0&1
 \end{array}\right)\in O(1,n+1).
 \end{equation}
Moreover, it is not hard to see that the future light cone at $o^{(1)}$ is
\begin{equation}
\mathcal L_+^{(1)}=\{(t,x)\in dS_{n+1}\ |\ x_{n+1}=1,t>0\}.
\end{equation}
The $n+1$ dimensional anti-de Sitter spacetime $AdS_{n+1}$ can be defined as the hypersurface
\begin{equation}
-t_1^2-t_2^2+x_1^2+x_2^2+\cdots+x_n^2=-1
\end{equation}
of $\R^{2,n}$ equipped with the induced metric. The group $O(2,n)$ acts on
$AdS_{n+1}$ transitively. Let $o^{(-1)}=(1,0,0,\cdots,0)$. Then, the isotropy group at $o^{(-1)}$ of the action is also $O(1,n)$ by identifying $A\in O(1,n)$ with
\begin{equation}\label{eqn-A--1}
 A^{(-1)}=\left(\begin{array}{ll}1&0\\0&A
 \end{array}\right)\in O(2,n).
 \end{equation}
 It is not hard to see that the future light cone of $AdS_{n+1}$ at $o^{(-1)}$ is
 \begin{equation}
\mathcal L_+^{(-1)}=\{(t,x)\in AdS_{n+1}\ |\ t_1=1,t_2>0\}.
\end{equation}

The main result of this paper is as follows.
\begin{thm}\label{thm-main}
\begin{enumerate}
\item a connected Riemannian manifold $(M,g)$ can be isometrically immersed into $\mathcal L_+^{(k)}$ for $k=0,\pm 1$ if and only if $(M,g)$ can be conformally immersed into $\sph^{n-1}$ equipped with the standard metric;
\item a closed connected Riemannian manifold $(M,g)$ of dimension $n-1$  with $n\geq 3$ can be isometrically immersed into $\mathcal L_+^{(k)}$ for $k=0,\pm 1$ if and only if $(M,g)$ is conformally diffeomorphic to $\sph^{n-1}$ equipped with the standard metric. Moreover, isometric immersion of $(M,g)$ into $\mathcal L_+^{(k)}$ is rigid.  More precisely, let  $\varphi_1:M\to \mathcal L_+^{(k)}$ and $\varphi_2:M\to \mathcal L_+^{(k)}$ be two isometric immersions. Then, there is a unique  $\tau\in O_+(1,n)$ such that $\varphi_2=\tau^{(k)}\circ\varphi_1$ on $M$;
\item isometric immersion of a connected $(n-1)$-dimensional Riemannian manifold $(M,g)$ (not necessarily closed) into $\mathcal L_+^{(k)}$ for $k=0,\pm 1$  is rigid with $n\geq 4$. More precisely, let  $\varphi_1:M\to \mathcal L_+^{(k)}$ and $\varphi_2:M\to \mathcal L_+^{(k)}$ be two isometric immersions. Then, there is a unique  $\tau\in O_+(1,n)$ such that $\varphi_2=\tau^{(k)}\circ\varphi_1$ on $M$.
\end{enumerate}
Here $O_+(1,n)$ is the group of Lorentz transformations preserving time direction, $\tau^{(0)}=\tau$ and the definitions of $\tau^{(k)}$ for $k=1$ and $k=-1$ are \eqref{eqn-A-1} and \eqref{eqn-A--1} respectively.
\end{thm}

The result (1) in Theorem \ref{thm-main}  can be viewed as an extension of Brinkmann's result since the standard metric on $\sph^{n-1}$ is locally conformally flat. It was also obtained in \cite{AD} for Minkowski spacetime. The result (2) in Theorem \ref{thm-main} was also obtained in \cite{AD,PR} for Minkowski spacetime.

It may be worth to note that (2) of Theorem \ref{thm-main} is not true for $n=2$. For example, let
$$M=\{(2\cos\theta,2\sin\theta)\ |\ \theta\in [0,2\pi]\}$$
 be a circle of radius 2. Let
 $$\varphi_1(2\cos\theta,2\sin\theta)=(2,2\cos\theta,2\sin\theta)$$
 and
 $$\varphi_2(2\cos\theta,2\sin\theta)=(1,\cos(2\theta),\sin(2\theta)).$$
 It is clear that $\varphi_1$ and $\varphi_2$ are both isometric immersions of $M$ into the light cone of $\R^{1,2}$. However, there is no Lorentz transformation $\tau$ such that $\varphi_2=\tau\circ\varphi_1$. Our result indicates that this kind of phenomenon never happens for $n\geq 3$.

 Moreover, the result (3) of Theorem \ref{thm-main} is not true for $n=3$. For example, let $M$ be the unit disk of $\C$. Let
 $$\ol\varphi_1:M\to \C\cup\{\infty\}=\sph^2$$
  be the natural inclusion and
 $$\ol\varphi_2:M\to \C\cup\{\infty\}=\sph^2$$
be a holomorphic embedding such that $\varphi_2(M)$ is the unit square in $\C$. The existence of such a holomorphic embedding is guaranteed by the Riemann mapping theorem. By Lemma \ref{lem-main} in the next section, we know that there are isometric embeddings $\varphi_1$  and $\varphi_2$ of $M$ into the future light cone of $\R^{1,3}$, such that $\pi\circ\varphi_i=\ol\varphi_i$ for $i=1,2$. However, because there is no conformal transformation of $\sph^2$ that transforms the unit disk to the unit square, there is no Lorentz transformation $\tau$ of $\R^{1,3}$ such that $\tau\circ\varphi_1=\varphi_2$. The difference of the cases $n=3$ and $n\geq 4$ is due to that Liouville's theorem (see \cite[Theorem A.3.7 \& Corollary A.3.8]{BP}) is not true for $n=3$.

Furthermore, by a similar argument of the proof of Theorem \ref{thm-main}, we can show that any isometry of the future light cone can be extended as an isometry of the whole spacetime for Minkowski, de Sitter and anti-de Sitter spacetimes.
\begin{thm}\label{thm-iso-cone}
Let $\varphi:\mathcal L_+^{(k)}\to \mathcal L_+^{(k)}$ be an isometry for $k=0,\pm1$ with $n\geq 3$. Then, there is a unique Lorentz transformation $\tau\in O_+(1,n)$ such that $\tau^{(k)}|_{\mathcal L_+^{(k)}}=\varphi$. Here $\tau^{(0)}=\tau$ and the definitions of $\tau^{(k)}$ for $k=1$ and $k=-1$ are \eqref{eqn-A-1} and \eqref{eqn-A--1} respectively.
\end{thm}
Note that, similar as before, Theorem \ref{thm-iso-cone} is not true for $n=2$.

Because
\begin{equation}\label{eqn-cone-1}
\mathcal L_+^{(1)}=\{(t,x)\in dS_{n+1}\ |\ -t^2+x_1^2+\cdots x_n^2=0,\ x_{n+1}=1,\ t>0  \}
\end{equation}
and
\begin{equation}\label{eqn-cone--1}
\mathcal L_+^{(-1)}=\{(t_1,t_2,x)\in AdS_{n+1}\ |\ t_1=1,\ -t_2^2+x_1^2+\cdots x_n^2=0, t_2>0  \},
\end{equation}
the proof of Theorem \ref{thm-main} can be reduced to the proof of the Minkowski case. So, the proofs of (1) and (2) in Theorem \ref{thm-main} are similar to that in \cite{AD,PR} for Minkowski spacetime. Since the proofs are short, we will also include them for completeness. The proofs of (3) in Theorem \ref{thm-main} and Theorem \ref{thm-iso-cone} are based on a simple application of the conformal structure on the null infinity (see \cite[\S 18.5]{Pe}).
\section{Rigidity of isometric immersions}

We will simply write the future light cone of $\R^{1,n}$ at the origin as $\mathcal L_+$. It admits a natural map $\pi$ to $\sph^{n-1}$ defined as
\begin{equation}
\pi(t,x)=\frac{x}{t}.
\end{equation}
In fact, $\sph^{n-1}$ can be viewed as the space of future directed light rays. So, it can be considered as the null infinity physically. The following lemma is a key ingredient in our proof of Theorem \ref{thm-main}.
\begin{lem}\label{lem-main}
Let $(M^m,g)$ be a Riemannian manifold and $\varphi:M\to \mathcal L_+$ be an isometric immersion. Then
\begin{enumerate}
\item $$(\pi\circ\varphi)^*g_{\sph^{n-1}}=t^{-2}g.$$
Here $g_{\sph^{n-1}}$ is the standard metric of $\sph^{n-1}$ and suppose that $\varphi=(t,x)$;
\item conversely, if $\psi:M\to \sph^{n-1}$ is a conformal immersion such that
\begin{equation}
\psi^*g_{\sph^{n-1}}=\lambda^{-2}g
\end{equation}
where $\lambda$ is a positive smooth function on $M$. Then, $\varphi=(\lambda,\lambda\psi):M\to\mathcal L_+$ is an isometric immersion;
\item let $\varphi_1$ and $\varphi_2$ be two isometric immersions of $M$ into $\mathcal L_+$ such that $\pi\circ\varphi_1=\pi\circ\varphi_2$. Then $\varphi_1=\varphi_2$.
\end{enumerate}
\end{lem}
\begin{proof}
\begin{enumerate}
\item Let $(y_1,y_2,\cdots,y_{m})$ be a local coordinate of $M$. Because $\varphi$ is an isometry, we have
\begin{equation}
\begin{split}
g\left(\frac{\p}{\p y_i},\frac{\p}{\p y_j}\right)=&\vv<\varphi_*\frac{\p}{\p y_i},\varphi_*\frac{\p}{\p y_j}>\\
=&\vv<\frac{\p t}{\p y_i}\frac{\p}{\p t}+\sum_{k=1}^n\frac{\p x_k}{\p y_i}\frac{\p}{\p x_k},\frac{\p t}{\p y_j}\frac{\p}{\p t}+\sum_{l=1}^n\frac{\p x_l}{\p y_j}\frac{\p}{\p x_l}>\\
=&-\frac{\p t}{\p y_i}\frac{\p t}{\p y_j}+\sum_{k=1}^n\frac{\p x_k}{\p y_i}\frac{\p x_k}{\p y_j}.
\end{split}
\end{equation}
On the other hand, let $(z_1,z_2,\cdots,z_n)$ be natural coordinates of $\R^n$. Then, along the map $\pi:\mathcal L_+\to \sph^{n-1}$, $z=\frac{x}{t}$. Hence,
\begin{equation}
\begin{split}
&(\pi\circ\varphi)^*g_{\sph^{n-1}}\left(\frac{\p}{\p y_i},\frac{\p}{\p y_j}\right)\\
=&\vv<\sum_{k=1}^n\frac{\p z_k}{\p y_i}\frac{\p}{\p z_k},\sum_{l=1}^n\frac{\p z_l}{\p y_j}\frac{\p}{\p z_l}>\\
=&\sum_{k=1}^n\left(-\frac{x_k}{t^2}\frac{\p t}{\p y_i}+\frac{1}{t}\frac{\p x_k}{\p y_i}\right)\left(-\frac{x_k}{t^2}\frac{\p t}{\p y_j}+\frac{1}{t}\frac{\p x_k}{\p y_j}\right)\\
=&\frac{1}{t^2}\left(-\frac{\p t}{\p y_i}\frac{\p t}{\p y_j}+\sum_{k=1}^n\frac{\p x_k}{\p y_i}\frac{\p x_k}{\p y_j}\right)\\
=&\frac{1}{t^2}g\left(\frac{\p}{\p y_i},\frac{\p}{\p y_j}\right).
\end{split}
\end{equation}
Here, we have used that $t^2=\|x\|^2$. This completes the proof of (1).
\item The proof of (2) is similar to that of (1) by direct computation.
\item Let $\varphi_1=(t,x)$ and $\varphi_2=(\tilde t,\tilde x)$. By (1) and that $\pi\circ\varphi_1=\pi\circ\varphi_2$, we have $t=\tilde t$ and $\frac{x}{t}=\frac{\tilde x}{\tilde t}$. Hence $\varphi_1=\varphi_2$.
\end{enumerate}
\end{proof}

It is clear that $\tau(\mathcal L_+)=\mathcal L_+$ for any $\tau\in O_+(1,n)$, and $\tau$ naturally descends to a map $\ol\tau:\sph^{n-1}\to\sph^{n-1}$ along $\pi: \mathcal L_+\to \sph^{n-1}$ for any $\tau\in O_+(1,n)$ such that
 \begin{equation}
 \pi\circ\tau=\ol\tau\circ\pi.
 \end{equation}
 By a direct computation, one can find that $\ol \tau$ is in fact a conformal transformation of $\sph^{n-1}$ equipped with standard metric. Conversely, any conformal transformation of $\sph^{n-1}$ can be obtained in this way for $n\geq 3$. More precisely, the group homomorphism from $O_+(1,n)$ to ${\rm Conf}(\sph^{n-1})$ by sending $\tau$ to $\ol\tau$ is indeed an isomorphism for $n\geq 3$, where ${\rm Conf}(\sph^{n-1})$ is the group of conformal transformation on $\sph^{n-1}$ equipped with standard metric. The proof of this classical isomorphism can be obtained by direct computation and by that ${\rm Conf}(\sph^{n-1})$ is generated by dilations, orthonormal transformations, translations and inversions of $\R^{n-1}\cup\{\infty\}$ when $\sph^{n-1}$ is identified as $\R^{n-1}\cup\{\infty\}$ via stereographic projection when $n\geq 3$ (see \cite[Corollary A.3.8]{BP}). Some related discussions can be found in \cite[\S 5]{Sl}. A detailed proof of this fact can be found in the appendix.

The isomorphism of $O_+(1,n)$ and ${\rm Conf}(\sph^{n-1})$ is another key ingredient in our proof of Theorem \ref{thm-main}. Physically, it indicates that there is a natural conformal structure on the null infinity induced by the spacetime (see \cite[\S 18.5]{Pe}).
\begin{proof}[Proof of Theorem \ref{thm-main}]Because of \eqref{eqn-cone-1} and \eqref{eqn-cone--1}, we only need to prove Theorem \ref{thm-main} for Minkowski spacetime.
\begin{enumerate}
\item It is a direct corollary of (1) and (2) in Lemma \ref{lem-main}.
\item By (1) of Lemma \ref{lem-main}, $\pi\circ\varphi:M\to \sph^{n-1}$ is a conformal immersion. Since $M$ is a closed manifold of dimension $n-1$, $\pi\circ\varphi$ is indeed a covering map. Note that $\sph^{n-1}$ is simply connected for $n\geq 3$. So, $\pi\circ\varphi$ is in fact a conformal diffeomorphism. Conversely, any conformal diffeomorphism from $M$ to $\sph^{n-1}$ induces an isometric embedding from $M$ to $\mathcal L_+$ by (2) of Lemma \ref{lem-main}.

    Moreover, from the above, $\pi\circ\varphi_i:M\to\sph^{n-1}$ is a conformal diffeomorphism for $i=1,2$. So, $\pi\circ\varphi_2\circ(\pi\circ\varphi_1)^{-1}\in{\rm Conf}(\sph^{n-1})$. By the isomorphism of $O_+(1,n)$ to ${\rm Conf}(\sph^{n-1})$ for $n\geq 3$, there is a Lorentz transformation $\tau$ such that
    \begin{equation}
    \ol\tau=\pi\circ\varphi_2\circ(\pi\circ\varphi_1)^{-1}
    \end{equation}
    which implies that
    \begin{equation}
      \pi\circ(\tau\circ\varphi_1)=\pi\circ\varphi_2.
    \end{equation}
By (3) of Lemma \ref{lem-main}, we have $\tau\circ\varphi_1=\varphi_2$.
    \item For each $p\in M$, let $U_p$ be a connected open neighborhood of $p$ such that
        $$\pi\circ \varphi_1|_{U_p}:U_p\to \pi\circ\varphi_1(U_p)\subset\sph^{n-1}$$ and
        $$\pi\circ \varphi_2|_{U_p}:U_p\to\pi\circ \varphi_2(U_p)\subset\sph^{n-1}$$ are both diffeomorphisms. By (1) of Lemma \ref{lem-main},
    $$\pi\circ \varphi_2|_{U_p}\circ(\pi\circ \varphi_1|_{U_p})^{-1}:\pi\circ\varphi_1(U_p)\to\pi\circ \varphi_2(U_p)$$ is a conformal diffeomorphism. By Liouville's theorem (see \cite[Theorem A.3.7 \& Corollary A.3.8]{BP}) and the isomorphism of $O_+(1,n)$ to ${\rm Conf}(\sph^{n-1})$, there is a unique $\tau_p\in O_+(1,n)$, such that
    \begin{equation}
      \ol{\tau_p}=\pi\circ \varphi_2|_{U_p}\circ(\pi\circ \varphi_1|_{U_p})^{-1}
    \end{equation}
    on $\pi\circ\varphi_1(U_p)$. So
    \begin{equation}
    \pi\circ\tau_p\circ\varphi_1=\pi\circ\varphi_2
    \end{equation}
    on $U_p$. This implies that ${\tau_p}\circ\varphi_1=\varphi_2$ on $U_p$ by (3) of Lemma \ref{lem-main}. By uniqueness of $\tau_p$, we know that the map $p\mapsto \tau_p$ from $M$ to $O_+(1,n)$ is locally constant. Since $M$ is connected, the map $p\mapsto \tau_p$ from $M$ to $O_+(1,n)$ is constant. This completes the proof of (3).
\end{enumerate}

\end{proof}
\begin{rem}
In the proof of (3) in Theorem \ref{thm-main}, we have used the fact that two conformal transformations of $\sph^{n-1}$ are identical on $\sph^{n-1}$ if they are identical on some open subset for $n\geq 3$. This comes from that any conformal transformation of $\sph^{n-1}$ is real analytic since the generators of ${\rm Conf}(\sph^{n-1})$ are all real analytic for $n\geq 3$.
\end{rem}

Next, we come to prove Theorem \ref{thm-iso-cone}.
\begin{proof}[Proof of Theorem \ref{thm-iso-cone}]
Similarly as before, because of \eqref{eqn-cone-1} and \eqref{eqn-cone--1}, we only need to prove it for Minkowski spacetime. Note that the future light cone $\mathcal L_+$ of $\R^{1,n}$ can be identified with $\R_+\times\sph^{n-1}$ equipped with the degenerate metric $t^2g_{\sph^{n-1}}$ by
\begin{equation}
\iota:\R_+\times \sph^{n-1}\to \R^{1,n}
\end{equation}
with $\iota(t,x)=(t,tx)$. Let
\begin{equation}
\varphi:\R_+\times \sph^{n-1}\to\R_+\times \sph^{n-1}
\end{equation}
be an isometry of the light cone. Suppose that $\varphi(t,x)=(f(t,x),\ol\varphi(t,x))$.
Then,
\begin{equation}
0=\vv<\frac{\p}{\p t},\frac{\p}{\p t}>=\vv<\varphi_*\frac{\p}{\p t},\varphi_*\frac{\p}{\p t}>=\left\|\frac{\p\ol\varphi}{\p t}\right\|^2.
\end{equation}
So, $\ol\varphi$ is independent of $t$ and
\begin{equation}
t^2g_{\sph^{n-1}}=\varphi^*(t^2g_{\sph^{n-1}})=f^2\ol\varphi^*g_{\sph^{n-1}}.
\end{equation}
So
\begin{equation}\label{eqn-ol-phi}
  \ol\varphi^*g_{\sph^{n-1}}=\frac{t^2}{f^2}g_{\sph^{n-1}}.
\end{equation}
This means that $\varphi$ descends to a conformal transformation $\ol\varphi$ of $\sph^{n-1}$ such that
$$\pi\circ\varphi=\ol\varphi\circ\pi$$
 on $\mathcal L_+$. By the isomorphism of $O_+(1,n)$ to $\conf(\sph^{n-1})$, there is a unique $\tau\in O_+(1,n)$, such that
\begin{equation}
\ol\tau=\ol\varphi.
\end{equation}
Suppose that $\ol\tau$ as an isometry of the light cone $\R_+\times\sph^{n-1}$ can be written as $\ol\tau=(\tilde f,\ol\varphi)$. Then, by \eqref{eqn-ol-phi}, we know that $\tilde f=f$. So,
\begin{equation}
\tau|_{\mathcal L_+}=\varphi.
\end{equation}

\end{proof}
\section{Appendix}
In the Appendix, we give a proof of the isomorphism of $O_+(1,n)$ and ${\rm Conf}(\sph^{n-1})$.
\begin{thm}
Let $\tau\in O_+(1,n)$. Then the induced map $\ol\tau\in {\rm Conf}(\sph^{n-1})$. Moreover, the map sending $\tau$ to $\ol\tau$ is a group isomorphism for $n\geq 3$.
\end{thm}
\begin{proof} Note that $\ol\tau\in {\rm Conf}(\sph^{n-1})$ has been shown in the proof of Theorem \ref{thm-iso-cone}. So, we only need to show that the map sending $\tau$ to $\ol\tau$ is a group isomorphism.

Let $\tau=\left(\begin{array}{cc}a&u^T\\v&A
\end{array}\right)\in O_+(1,n)$ where $a\in \R_+$, $u,v\in\R^n$ and $A\in \R^{n\times n}$. Then, $\tau\in O_+(1,n)$ is equivalent to
\begin{equation}\label{eqn-Lorentz}
\left\{\begin{array}{l}a^2-\|v\|^2=1\\
A^Tv-au=0\\
A^TA-uu^T=I_n\\
a>0.
\end{array}\right.
\end{equation}
It is clear that
\begin{equation}
\ol\tau(z)=\frac{Az+v}{u^Tz+a}
\end{equation}
with $z\in \sph^{n-1}\subset \R^n$.

By Liouville's theorem (see \cite[Proposition A.3.4 \& Theorem A.3.7]{BP}), the group $\conf(\sph^{n-1})$ is generated by
\begin{enumerate}
\item dilation: $w\mapsto \lambda w$ with $\lambda\neq 0$;
\item orthonormal transformation: $w\mapsto B w$ with $B\in O(n-1)$;
\item inversion: $w\mapsto \frac{w-w_0}{\|w-w_0\|^2}$ with $w_0\in \R^{n-1}$;
\item translation: $w\mapsto w+b$ with $b\in \R^{n-1}$
\end{enumerate}
when identifying $\sph^{n-1}$ with $\R^{n-1}\cup\{\infty\}$ by stereographic projection:
\begin{equation}
    w=\frac{z'}{1-z_1}.
\end{equation}
Here, we write $z$ as $(z_1,z')$.
So, to show surjectivity of the map, we only need to show that each of the generators has a preimage.
\begin{enumerate}
\item Dilation. By direct computation using the stereographic projection, the dilation $\tilde w=\lambda w$ on $\R^{n-1}$ is corresponding to
    \begin{equation}
    \left\{\begin{array}{l}\tilde z_1=\frac{(\lambda^2+1)z_1+\lambda^2-1}{(\lambda^2-1)z_1+\lambda^2+1}=\frac{(\lambda+1/\lambda)z_1+\lambda-1/\lambda}{(\lambda-1/\lambda)z_1+\lambda+1/\lambda}\\
    \tilde z'=\frac{2\lambda z'}{(\lambda^2-1)z_1+(\lambda^2+1)}=\frac{2 z'}{(\lambda-1/\lambda)z_1+\lambda+1/\lambda}   \end{array}\right.
    \end{equation}
    on $\sph^{n-1}$. Let $u=\frac{1}{2}(\lambda-\frac{1}{\lambda})e_1$, $a=\frac{1}{2}(\lambda+\frac{1}{\lambda})$, $v=\frac{1}{2}(\lambda-\frac{1}{\lambda})e_1$ and $A=\diag (\frac{1}{2}(\lambda+\lambda^{-1}),1,\cdots,1)$ when $\lambda>0$. Then, $a,u,v,A$ satisfy \eqref{eqn-Lorentz} and
\begin{equation}
\tilde z=\frac{Az+v}{a+u^Tz}.
\end{equation}
So, we find the preimage of the dilation. When $\lambda<0$, one could just replace $u,v,a,A$ by the negatives of the previous $u,v,a,A$.
\item Orthonormal transformation. The orthonormal transformation $\tilde w=Bw$ is corresponding to
    \begin{equation}
    \left\{\begin{array}{l}\tilde z_1=z_1\\\tilde z'=Bz'.
    \end{array}\right.    \end{equation}
    Let $a=1,u=v=0$ and $A=\diag(1,B)$. Then $a,u,v,A$ satisfy \eqref{eqn-Lorentz} and
    \begin{equation}
\tilde z=\frac{Az+v}{a+u^Tz}.
\end{equation}
Thus we find the preimage of the orthonormal transformation.
\item Inversion. The inversion $\tilde w=\frac{w-w_0}{\|w-w_0\|^2}$  is corresponding to
    \begin{equation}
    \left\{\begin{array}{l}\tilde z_1=\frac{(-1+\|w_0\|^2/2)z_1+w_0^Tz'-\|w_0\|^2/2}{-\|w_0\|^2z_1/2-w_0^Tz'+1+\|w_0\|^2/2}\\
    \tilde z'=\frac{z_1w_0+z'-w_0}{-\|w_0\|^2z_1/2-w_0^Tz'+1+\|w_0\|^2/2}.
    \end{array}\right.
    \end{equation}
    By letting $a=1+\|w_0\|^2/2$, $u=\left(\begin{array}{c}-\frac{\|w_0\|^2}{2}\\-w_0
    \end{array}\right)$, $v=\left(\begin{array}{c}\frac{-\|w_0\|^2}{2}\\-w_0
    \end{array}\right)$ and $A=\left(\begin{array}{cc}-1+\frac{\|w_0\|^2}{2}&w_0^T\\w_0&I_{n-1}
    \end{array}\right)$, we have
    \begin{equation}
    \tilde z=\frac{Az+v}{u^Tz+a}
    \end{equation}
    and $a,u,v,A$ satisfying \eqref{eqn-Lorentz}. Thus, we find the preimage of the inversion.
\item Translation. The translation $\tilde w=w+b$ is corresponding to
    \begin{equation}
    \left\{\begin{array}{l}
    \tilde z_1=\frac{(1-\frac{\|b\|^2}{2})z_1+b^Tz'+\frac{\|b\|^2}{2}}{-\frac{\|b\|^2}{2}z_1+b^Tz'+(1+\frac{\|b\|^2}{2})}\\
    \tilde z'=\frac{z'+(1-z_1)b}{-\frac{\|b^2\|}{2}z_1+b^Tz'+(1+\frac{\|b\|^2}{2})}.
    \end{array}\right.\end{equation}
    By letting $a=1+\|b\|^2/2$, $u=\left(\begin{array}{c}-\frac{\|b\|^2}{2}\\b
    \end{array}\right)$, $v=\left(\begin{array}{c}\frac{\|b\|^2}{2}\\b
    \end{array}\right)$ and $A=\left(\begin{array}{cc}1-\frac{\|b\|^2}{2}&b^T\\-b&I_{n-1}
    \end{array}\right)$, we have
    \begin{equation}
    \tilde z=\frac{Az+v}{u^Tz+a}
    \end{equation}
    and $a,u,v,A$ satisfying \eqref{eqn-Lorentz}. Thus, we find the preimage of the translation.
\end{enumerate}
Injectivity of the map can be shown easily by checking that the kernel of the map is trivial. This completes the proof of the theorem.
\end{proof}

\end{document}